\documentclass{amsart}
\newtheorem*{theorem3.10}{Theorem 3.10}
\newtheorem{theorem}{Theorem}[section]
\newtheorem{lemma}[theorem]{Lemma}
\newtheorem{corollary}[theorem]{Corollary}
\theoremstyle{definition}
\newtheorem{definition}[theorem]{Definition}
\newtheorem{example}[theorem]{Example}
\newtheorem{setting}[theorem]{Setting}

\theoremstyle{remark}
\newtheorem{remark}[theorem]{Remark}

\begin{document}

\title{A Class of Hilbert Series and the Strong Lefschetz Property}

\author{Melissa Lindsey}
\address{Department of Mathematics, Purdue University, West Lafayette, IN 47907-2067}
\email{lindsey9@math.purdue.edu}
\thanks{}

\subjclass[2000]{Primary 13A02; 13C05}

\date{}

\keywords{Hilbert Series, Strong Lefschetz Property}

\commby{Bernd Ulrich}

\begin{abstract}
We determine the class of Hilbert series $\mathcal H$ so that if $M$ is a finitely generated zero-dimensional $R$-graded module with the strong Lefschetz property, then $M\otimes_k k[y]/(y^m)$ has the strong Lefschetz property for an indeterminate $y$ and all positive integers $m$ if and only if the Hilbert series of $M$ is in $\mathcal{H}$. Given two finite graded $R$-modules $M$ and $N$ with the strong Lefschetz property, we determine sufficient conditions in order that $M\otimes_kN$ has the strong Lefschetz property.
\end{abstract}

\maketitle

\section{Introduction}
Let $k$ be a field of characteristic zero, and let $R=k[x_1,\dots,x_n]=\bigoplus_{d\geq 0}R_d$ be a standard graded polynomial ring in $n$ variables over $k$. A graded $R$-module $M=\bigoplus_{d\in \mathbb{Z}}M_d$ is said to have the {\em{strong Lefschetz property}} if there exists a linear form $l\in R_1$ such that the $k$-linear map $l^a:M_d\mapsto M_{d+a}$ has full rank for all nonnegative integers $a$ and $d$. In other words, for each $a$ and $d$, the map $l^a$ is either injective or surjective. In this situation $l$ is called a {\em{strong Lefschetz element}} for $M$.

A number of people have worked on classifying the rings that have the strong Lefschetz property, for example see the work of Harima and Watanabe \cite{HW} and Migliore and Miro-Roig\cite{MM}.  In \cite{HMNW}  a characterization of the Hilbert functions that can occur for $k$-algebras with the strong Lefschetz property is given and in \cite{ZZ} the authors determine which Hilbert functions force the strong Lefschetz property. Stanley used the hard Lefschetz theorem \cite[Lemma 2.3]{S} to show that if an ideal $I\subset R$ is generated by a monomial regular sequence, then $R/I$ has the strong Lefschetz property. Watanabe later gave another proof of the same fact using representation theory \cite[Corollary 3.5]{W}. Reid-Roberts-Roitman \cite[Theorem 10]{RRR} gave the first purely algebraic proof of this result. As a corollary of our work we give another algebraic proof of this result as described in Remark \ref{monomialcompleteintersection}. It is an open question, even in codimension three, as to whether all complete intersections have the strong Lefschetz property. An interesting related question is whether $R/I$ has the strong Lefschetz property when $I\subset R$ is generated by general forms. Fr\"{o}berg's conjecture on the Hilbert function of an ideal generated by general forms follows for $R/I$ if $R/I$ has the strong Lefschetz property, see \cite{F}, \cite[Conjecture A]{P2} and \cite[Conjecture 1.2]{CP}.

Two other invariants of an Artinian local ring $A$ with maximal ideal $\mathfrak{m}$ are its \em{Dilworth number}, $\max\{\mu({\mathfrak{m}}^d)~|~d\geq 1\}$, and its \em{Sperner number}, $\max\{\mu(I)~|~I\subset A\}$. Watanabe proves that ``most" Gorenstein rings have the strong Lefschetz property \cite[Example 3.9]{W} and also proves that if $A$ has the strong Lefschetz property, then the Dilworth number of $A$ is equal to the Sperner number of $A$ \cite[Proposition 3.5]{W}.  An interesting example constructed by Ikeda  \cite[Example 4.4]{I} with Hilbert function $(1,4,10,10,4,1)$ provides an example of a Gorenstein ring with Sperner number $10$ and Dilworth number $11$, thus giving an example of a Gorenstein ring that does not have the strong Lefschetz property.

Herzog and Popescu \cite{HP} show that if $M$ is a standard graded Artinian Gorenstein $k$-algebra with the strong Lefschetz property, then $M[y]/(y^m)$ has the strong Lefschetz property for $y$ an indeterminate. In this paper we further explore the relationship between the strong Lefschetz property and these types of extensions through the study of the decomposition of the module over a PID. We introduce below a concept that we call {\em{almost centered}} (Definition \ref{almostcentered}) that involves a partial order on the summands in this decomposition. We determine that in order for the extension to have the strong Lefschetz property its decomposition with respect to a strong Lefschetz element must be almost centered.

A main result of this paper describes the class $\mathcal{H}$ of Hilbert series for which the following theorem holds:

\begin{theorem3.10}\label{qh}
Let $l\in R_1$ be a strong Lefschetz element for $M$ and $S=k[l]$. Then the following are equivalent:
\begin{itemize}
\item[(i)] $H_M(t)\in \mathcal{H}$,
\item[(ii)] $l+y\in R[y]_1$ is a strong Leschetz element of $M\otimes_k k[y]/(y^m)$ for all $m\geq0$,
\item[(iii)] The decomposition of $M$ as an $S$-module is almost centered.
\end{itemize}
\end{theorem3.10}

To prove Theorem \ref{class} we analyze the decomposition of $M$ into cyclic $k[l]$-modules where $l\in R_1$. First we prove in Theorem \ref{extends} the equivalence of (ii) and (iii) of Theorem \ref{class}. Given a module $M$ with the strong Lefschetz property, the almost centered condition provides a way of finding more modules with the strong Lefschetz property.  
 With additional hypotheses, we show in Theorem \ref{halmostclosedundertensor} and Corollary \ref{symmetrictensor} that the  tensor product of finite graded $R$-modules with the strong Lefschetz and almost centered properties again has these properties. 

The almost centered property of Definition \ref{almostcentered} links the Hilbert series of $M$ to the strong Lefschetz property extending
from $M$  to $M\otimes_k k[y]/(y^m)$ for all $m\geq0$.  Let $H_M(t)=\sum_{i=0}^ph^it^i$ denote the Hilbert series of $M$ where $h_i:=\dim M_i$ and $p$ is the socle degree (or postulation number) of $M$. We define the {\em socle degree} of $M$ as the largest nonnegative integer for which the Hilbert function of $M$ differs from the Hilbert polynomial of $M$. 

\begin{definition}\label{h}
We show that the class of Hilbert series that satisfies Theorem \ref{class} is precisely
\begin{eqnarray*}
{\mathcal{H}}=\{H_M(t)&|&h_{i-1}\leq h_{p-i}\leq h_i \text{ for all } 1\leq i \leq \lfloor \frac{p}{2} \rfloor \text{ or }\\
&& h_{p-i+1}\leq h_{i}\leq h_{p-i} \text{ for all } 1\leq i \leq \lfloor \frac{p}{2} \rfloor \}.
\end{eqnarray*}
\end{definition}

\section{Setting and Preliminary Results}
We use the following notation throughout the paper:
\begin{setting}\label{2.1}
Let $k$ be a field of characteristic zero, and let $R=k[x_1,\dots,x_n]$ be a polynomial ring over $k$. Let $M$ be a finitely generated zero-dimensional $R$-graded module. For convenience of notation we assume that the first nonzero degree of $M$ is $M_0$. We may do this because the results hold for $M$ if and only if they hold for a shifted version of $M$. Let $S=k[l]$ where $l\in R_1$.
\end{setting}

We use the graded version of the fundamental theorem of modules over PID's to decompose $M$ into cyclic $k[l]$ modules, where $l\in R_1$. We record in Theorem \ref{PID} the version we will use. The proof follows from a similar argument to that of Lang \cite[Theorem 7.5]{L}.

\begin{theorem}\label{PID} 
Let $M$ and $R$ be as described in Setting \ref{2.1}, and $l\in R_1$. Set $S=k[l]$. Then there exists a degree zero $S$-module isomorphism $$M\cong \bigoplus_{\alpha \in \Lambda} S_{\alpha}$$ where $S_{\alpha}=S(-i_{\alpha})/(l^{d_{\alpha}})$ and $\Lambda$ is a finite set, $i_{\alpha}$ are nonnegative integers, and $d_{\alpha}$ are positive integers. 
\end{theorem}
%\begin{proof} Apply the argument given by Lang in \cite[Theorem 7.5]{L}, but modify the definition of independent \cite[page 150]{L} by requiring in the equation $$a_1y_1+\cdots +a_my_m=0,$$ with $a_i5\in S$ and $y_i \in M$ that the terms $a_iy_i$ are all homogeneous of the same degree.
%\end{proof}

\begin{example}\label{2dimensional}
Let $R=k[x_1,x_2]$ and $M=R/(x_1^3,x_2^5)$. Then $M$ is generated as an $S$-module by $\{1,x_2,x_2^2,x_2^3\}$, and its decomposition as an $S$-module given by Theorem \ref{PID} is $M\cong S/(l^7)\oplus S(-1)/(l^5)\oplus S(-2)/(l^3) \oplus S(-3)/(l)$.

More generally, if $M=R/(x_1^a, x_2^b)$.  Watanabe \cite{W} and Stanley \cite{S} both show that $l=x_1+x_2$ is a strong Lefschetz element for $M$. Considering $M$ as an $S=k[l]$ module, Thoerem \ref{PID} gives $M\cong \bigoplus_{h=0}^{\min\{a,b\}} S(-h)/(l^{a+b-1-2h})$.
\end{example}

For $l\in R_1$ (not necessarily a strong Lefschetz element), we write $S=k[l]$ and $M\cong \bigoplus_{\alpha \in \Lambda} S_{\alpha}$ for the decomposition of $M$ as an $S$-module given by Theorem \ref{PID}, where $S_{\alpha}=S(-i_{\alpha})/(l^{d_{\alpha}})$. 

\begin{remark} With notation as above, let $p$ denote the socle degree of $M\cong\bigoplus_{\alpha \in \Lambda}, S_{\alpha}$ where $S_{\alpha}=S(-i_{\alpha})/(l^{d_{\alpha}})$. Then for all $\alpha$, $i_{\alpha}\leq p$ and $d_{\alpha} \leq p+1$. This follows because $\dim M_d = 0$ for all $d>p$ and $\dim (S_{\alpha})_d \neq 0$ for all $i_{\alpha}\leq d \leq d_{\alpha}+i_{\alpha}-1$.
\end{remark}

\begin{lemma}\label{removingpieces}
Let $M$ and $S$ be as above and assume that $l\in R_1$ is a strong Lefschetz element for $M$. Then for all $\Gamma \subset \Lambda$ the $S$-module $\bigoplus_{\alpha \in \Gamma} S_{\alpha}$ has the strong Lefschetz property with strong Lefschetz element $l$.
\end{lemma}
\begin{proof}
If $l$ is a strong Lefschetz element for $M$, then for any fixed choice of $a$ and $d$, the map $l^a:(S_{\alpha})_d\to(S_{\alpha})_{d+a}$ is injective for all $\alpha \in \Lambda$ or is surjective for all $\alpha \in \Lambda$. Hence for any choice of $a$ and $d$,  $l^a:(\bigoplus_{\alpha \in \Gamma} S_{\alpha})_d\to (\bigoplus_{\alpha \in \Gamma} S_{\alpha})_{d+a}$ is injective or surjective, giving that $l$ is a strong Lefschetz element for $\bigoplus_{\alpha \in \Gamma} S_{\alpha}$.
\end{proof}

The decomposition of $M$ as a $k[l]$-module motivates the definition of a partial order on the summands of the decomposition.

\begin{definition}\label{order}
Let $M$, $S$, and $S_{\alpha}$ be as above. We define a partial order on $\{S_{\alpha}\}_{\alpha \in \Lambda}$ in the following way: for all $\beta$ and $\alpha$ in $\Lambda$,  $$S_{\beta} \preceq S_{\alpha} \iff i_{\alpha} \leq i_{\beta} \text{~~~and~~~} d_{\beta}+i_{\beta} \leq d_{\alpha}+i_{\alpha}.$$
When $S_{\beta}\preceq S_{\alpha}$, we define the difference between the starting degrees of $S_{\alpha}$ and $S_{\beta}$ to be $L_{\alpha \beta}:=i_{\beta}-i_{\alpha}$ and the difference between the ending degrees to be $R_{\alpha \beta}:=d_{\alpha}+i_{\alpha}-d_{\beta}-i_{\beta}$.
\end{definition}

The following lemma demonstrates a connection between the partial order $\preceq$ and the strong Lefschetz property. This connection plays an integral role in the proof of Theorem \ref{class}.

\begin{lemma}\label{totallyordered}
Let $M$, $S=k[l]$, and $S_{\alpha}$ be as above. Then $l$ is a strong Lefschetz element for $M$   if and only if $\{S_{\alpha}\}_{\alpha \in \Lambda}$ is totally ordered with respect to $\preceq$.
\end{lemma}

\begin{proof}Recall that for all $\alpha \in \Lambda$, $S_{\alpha}=S(-i_{\alpha})/(l^{d_{\alpha}})$.

Suppose that $l$ is a strong Lefschetz element for $M$. Let $\gamma$ and $\beta$ be in $\Lambda$. We may assume that $i_{\gamma} \leq i_{\beta}$. We show that either $S_{\beta} \preceq S_{\gamma}$ or $S_{\gamma} \preceq S_{\beta}$. Since $l$ is a strong Lefschetz element for $M$, Lemma \ref{removingpieces} implies that $l$ is a strong Lefschetz element for the $S$-module $S_{\beta}\bigoplus S_{\gamma}$. 

If $i_{\gamma} = i_{\beta}$, it follows immediately that $d_{\beta}+i_{\beta} \leq d_{\gamma}+i_	{\gamma}$ or $d_{\gamma}+i_{\gamma}\leq d_{\beta}+i_{\beta}$ which is equivalent to $S_{\beta}\preceq S_{\gamma}$ or $S_{\gamma}\preceq S_{\beta}$.

If $i_{\gamma} < i_{\beta}$, suppose that $S_{\beta}$ and $S_{\gamma}$ are not 		comparable. Then $d_{\gamma}+i_{\gamma}<d_{\beta}+i_{\beta}$. We find $a$ and $d$ so that the map $$l^a: (S_{\beta}\bigoplus S_{\gamma})_d \to (S_{\beta}\bigoplus S_{\gamma})_{d+a}$$ is injective but not surjective on one component and surjective but not injective on the other component, contradicting that $S_{\beta}\bigoplus S_{\gamma}$ has the strong Lefschetz property.

If $d_{\gamma}+i_{\gamma} \leq i_{\beta}$, set $a=i_{\beta}-i_{\gamma}-d_{\gamma}+1$ and $d=d_{\gamma}+i_{\gamma}-1$. Then\\ $(S_{\beta})_d=(S_{\gamma})_{d+1}=1$ and $(S_{\beta})_{d+1}=(S_{\gamma})_d=0$.

If $i_{\beta}\leq d_{\gamma}+i_{\gamma}-1$, set $a=\max\{d_{\beta}, d_{\gamma}\}$, and set $d=i_{\beta}-1$ when $a=d_{\beta}$ and $d=i_{\gamma}$ when $a=i_{\gamma}$. Then $(S_{\beta})_d=(S_{\gamma})_{d+1}=0$ and $(S_{\beta})_{d+1}=(S_{\gamma})_d=1$. 
	
Conversely, suppose that $\{S_{\alpha}\}_{\alpha \in \Lambda}$ is totally ordered with respect to $\preceq$ and assume that $l$ is not a strong Lefschetz element for $M$. Then there exists $a$ and $d$ so that the map $l^a:S_d\to S_{d+a}$ is neither injective or surjective. It follows that there exists $\beta$ and $\gamma$ in $\Lambda$, as well as $a$ and $d$, so that the map $$l^a:(S_{\beta})_d\to (S_{\beta})_{d+a}$$ is injective but not surjective and $$l^a:(S_{\gamma})_d\to (S_{\gamma})_{d+a}$$ is surjective but not injective. Since $\{S_{\alpha}\}_{\alpha \in \Lambda}$ is totally ordered, the two maps force $S_{\beta} \preceq S_{\gamma}$. This combined with the information about the maps injectivity and surjectivity force the inequalities $d_{\beta}+i_{\beta}\leq d_{\gamma} +i_{\gamma}\leq d+a \leq d_{\beta}+i_{\beta}-1,$ which is a contradiction.
\end{proof} 

\begin{remark}\label{largestelement}
With notation as in Lemma \ref{totallyordered}, if $M$ has the strong Lefschetz property and $l\in R_1$ is a strong Lefschetz element, then the decomposition of $M$ into cyclic modules is unique regardless of the strong Lefschetz element $l$ that is chosen. Also, the total ordering on $\{S_{\alpha}\}_{\alpha\in\Lambda}$ implies that $\{S_{\alpha}\}_{\alpha\in\Lambda}$ has a largest element, $S/(l^{p+1})$. In the case where $M$ is a cyclic $R$-module, there is only one copy of $S/(l^{p+1})$ making it the unique largest element of $\{S_{\alpha}\}_{\alpha\in\Lambda}$.
\end{remark}

\section{Main Result}
Given a module $M$ as in Setting \ref{2.1} where $l\in R_1$ is a strong Lefschetz element for $M$, we describe the decomposition of $M$ as an $S$-module if the strong Lefschetz property extends to $M\otimes_k k[y]/(y^m)$.

\begin{definition}\label{almostcentered}
For $l\in R_1$ a strong Lefschetz element of $M$ and $S=k[l]$, the decomposition of $M$ as an $S$-module $M\cong_S \bigoplus_{\alpha \in \Lambda} S_{\alpha}$ is said to be {\em{almost centered}} if whenever $\alpha$ and $\beta$ are in $\Lambda$, with $S_{\beta} \preceq S_{\alpha}$, then $|L_{\alpha \beta}-R_{\alpha \beta}|\leq 1$.
\end{definition}

\begin{remark}
The condition $|L_{\alpha \beta}-R_{\alpha \beta}|\leq 1$ in Definition \ref{almostcentered} says that $S_{\alpha}$ and $S_{\beta}$ are as centered as possible with respect to their starting and ending degrees. Notice that if $|L_{\alpha \beta}-R_{\alpha \beta}|=0$ for all $\alpha$ and $\beta$ in $\Lambda$, then $\dim M_d = \dim M_{p-d}$ for all $0\leq d \leq p$, hence the Hilbert function of $M$ is symmetric.
\end{remark}

\begin{theorem}\label{extends}Let $l\in R_1$ be a strong Lefschetz element for $M$, where $M$ is as in Setting \ref{2.1}. Then the $S[y]$ module $M\otimes_k k[y]/(y^m)$ has the strong Lefschetz property for all $m\geq0$ if and only if the decomposition of $M$ as an $S$-module is almost centered. In this case $l+y$ is a strong Lefschetz element of $M\otimes_k k[y]/(y^m)$
\end{theorem}
\begin{proof}
As an $S[y]$ module, $M\otimes_k k[y]/(y^m)\cong \bigoplus_{\alpha\in \Lambda} k[l,y](-i_{\alpha})/(l^{d_{\alpha}},y^m)$. Let $z=l+y$, and write $A=k[z]$. Given the decomposition of $M$ as an $S$-module, $M\cong \bigoplus_{\alpha \in \Lambda} S_{\alpha}$, we can use Example \ref{2dimensional} to decompose $M\otimes_k k[y]/(y^m)$ as an $A$ module: $$M\otimes_k k[y]/(y^m)\cong \bigoplus_{\alpha \in \Lambda} \bigoplus_{h=0}^{\min\{d_{\alpha},m\}} A(-(i_{\alpha}+h))/(z^{d_{\alpha}+m-2h-1}).$$ We will write $A_{\alpha h}$ for $A(-(i_{\alpha}+h))/(z^{d_{\alpha}+m-2h-1})$.

Suppose $M\otimes_k k[y]/(y^m)$ has the strong Lefschetz property for all positive $m$. Lemma \ref{totallyordered} gives that $\{A_{\alpha h}\}_{\alpha \in \Lambda, 0\leq h \leq \min\{d_{\alpha},m\}}$ is totally ordered for all $m$. Suppose that the decomposition of $M\cong \bigoplus_{\alpha \in \Lambda} S_{\alpha}$ is not almost centered. Then there exist $\alpha$ and $\beta$ in $\Lambda$ with $S_{\beta} \preceq S_{\alpha}$ such that $|L_{\alpha \beta}-R_{\alpha \beta}|\geq 2$.

Set $m=h=\min\{L_{\alpha \beta}+1,R_{\alpha \beta}+1\}$ and consider $$A_{\alpha h}=A(-(i_{\alpha}+h))/(z^{d_{\alpha}-h-1})~~\text{and}~~A_{\beta 0}=A(-i_{\beta})/(z^{d_{\beta}+h-1}).$$

If $m=h=L_{\alpha \beta}+1$, then $i_{\alpha}+h=i_{\beta}+1>i_{\beta}$, which implies $A_{\alpha h}\preceq A_{\beta 0}$ because $\{A_{\alpha h}\}_{\alpha \in \Lambda, 0\leq h \leq \min\{d_{\alpha},m\}}$ is totally ordered. Hence
\begin{eqnarray*}
d_{\alpha}-h-1+i_{\alpha}+h &\leq& d_{\beta}+h-1+i_{\beta} \\
d_{\alpha}+i_{\alpha}-d_{\beta}-i_{\beta} &\leq& h \\
R_{\alpha \beta} &\leq& L_{\alpha \beta}+1 \leq R_{\alpha \beta}+1
\end{eqnarray*}
which is a contradiction to $|L_{\alpha \beta}-R_{\alpha \beta}|\geq 2$.

If $m=h=R_{\alpha \beta}+1$, then $$i_{\alpha}+h=R_{\alpha \beta}-L_{\alpha \beta}+i_{\beta}+1\leq -2+i_{\beta}+1=i_{\beta}-1<i_{\beta},$$ which implies that $A_{\beta 0}\preceq A_{\alpha h}$. Therefore\begin{eqnarray*}
d_{\beta}+h-1+i_{\beta} &\leq& d_{\alpha}-h-1+i_{\alpha}+h\\
h&\leq&d_{\alpha}+i_{\alpha}-d_{\beta}-i_{\beta} \\
R_{\alpha \beta}+1&\leq&R_{\alpha \beta}
\end{eqnarray*} which is a contradiction.

Therefore the decomposition of $M$ as an $S$-module is almost centered whenever $M\otimes_k k[y]/(y^m)$ has the strong Lefschetz property for all $m$.

Conversely, suppose that the decomposition of $M$ as an $S$-module is almost centered and assume that $M\otimes_k k[y]/(y^m)$ does not have the strong Lefschetz property for all $m$. Then there exists an $m$ such that $\{A_{\alpha h}\}_{\alpha \in \Lambda, 0\leq h \leq \min\{d_{\alpha},m\}}$ is not totally ordered. Hence there exist $\alpha$, $\beta$, $0\leq h_{\alpha}\leq \min\{d_{\alpha},m\}$ and $0\leq h_{\beta}\leq \min\{d_{\beta},m\}$ so that $A_{\alpha h_{\alpha}}$ and $A_{\beta h_{\beta}}$ are not comparable under the partial order defined in Definition \ref{order}. We may assume that $i_{\alpha}+h_{\alpha}<i_{\beta}+h_{\beta}$. Then \begin{eqnarray*}
d_{\beta}+m-2h_{\beta}-1+i_{\beta}+h_{\beta}&>&d_{\alpha}+m-2h_{\alpha}-1+i_{\alpha}+h_{\alpha}\\
h_{\alpha}-h_{\beta} &>& d_{\alpha}+i_{\alpha}-d_{\beta}-i_{\beta}\\
0>(i_{\alpha}+h_{\alpha})-(i_{\beta}+h_{\beta})&>&d_{\alpha}+2i_{\alpha}-d_{\beta}-2i_{\beta}
\end{eqnarray*}
This implies that $-2\geq d_{\alpha}+2i_{\alpha}-d_{\beta}-2i_{\beta}$ which contradicts the assumption that the decomposition of $M$ is almost centered.
\end{proof}

Theorem \ref{extends} allows us to construct new modules with the strong Lefschetz property, in particular the modules $M\otimes_kk[y]/(y^m)$ where the decomposition of $M$ is almost centered. The decomposition of $M\otimes_kk[y]/(y^m)$ is also almost centered when the decomposition of $M$ is almost centered.

\begin{corollary}\label{tensorisac}
Let $l\in R_1$ be a strong Lefschetz element for the $R$-module $M$ and set $S=k[l]$. If the decomposition of $M$ as an $S$-module is almost centered, then the decomposition of $M\otimes_kk[y]/(y^m)$ as a $k[l+y]$ module is almost centered.
\end{corollary}
\begin{proof}
Using the notation of Theorem \ref{extends}, let $z=l+y$ and write $A=k[z]$. Given the decomposition of $M$ as an $S$-module, $M\cong \bigoplus_{\alpha \in \Lambda} S_{\alpha}$, we can use Example \ref{2dimensional} to decompose $M\otimes_k k[y]/(y^m)$ as an $A$ module: $$M\otimes_k k[y]/(y^m)\cong \bigoplus_{\alpha \in \Lambda} \bigoplus_{h=0}^{\min\{d_{\alpha},m\}} A(-(i_{\alpha}+h))/(z^{d_{\alpha}+m-2h-1}).$$

Theorem \ref{extends} tells us that the decomposition of $M\otimes_k k[y]/(y^m)$ is totally ordered. Suppose there exists $m$ such that the decomposition of $M\otimes_k k[y]/(y^m)$ is not almost centered. Then there exist $\alpha$, $\beta$, $h_{\alpha}$, and $h_{\beta}$ with $S_{\beta}\preceq S_{\alpha}$ so that
\begin{eqnarray*}
2&\leq& |d_{\alpha}+m-2h_{\alpha}-1+2i_{\alpha}+2h_{\alpha}-d_{\beta}-m+2h_{\beta}+1-2i_{\beta}-2h_{\beta}| \\
&=& |d_{\alpha}+2i_{\alpha}-d_{\beta}-2i_{\beta}| \\
&=& |L_{\alpha \beta}-R_{\alpha \beta}|,
\end{eqnarray*}but this contradicts the decomposition of $M$ being almost centered.
\end{proof}

\begin{theorem}\label{halmostclosedundertensor}
Let $l \in R_1$ and $y\in R_1$ be strong Lefschetz elements for the $R$-modules $M$ and $N$ respectively. If the Hilbert function of $M$ is symmetric and the decomposition of $N$ as a $k[y]$ module is almost centered, then $M\otimes_kN$ has the strong Lefschetz property and the decomposition of $M\otimes_kN$ as a $k[l+y]$ module is almost centered.
\end{theorem}
\begin{proof} Given $M\cong \bigoplus_{\alpha \in \Lambda}S_{\alpha}$ where $S=k[l]$ and $S_{\alpha}=S(-i_{\alpha})(l^{d_{\alpha}})$, notice that $|L_{\alpha \beta}-R_{\alpha \beta}|=0$ for all $\alpha$ and $\beta$ in $\Lambda$, because $M$ has the strong Lefschetz property and its Hilbert function is symmetric.

We may assume that $N\cong k[y](-j_m)/(y^m)\bigoplus k[y](-j_n)/(y^n)$ with $j_m\leq j_n$ and $n+j_n\leq m+j_m$, so  $$M\otimes N\cong (M\otimes k[y]/(y^m))\bigoplus(M\otimes k[y](-j)/(y^n)).$$ Theorem \ref{extends} implies that the decompositions of $M\otimes k[y]/(y^m)$ and $M\otimes k[y](-j)/(y^n)$ are totally ordered, and Corollary \ref{tensorisac} implies that the decompositions are almost centered. It remains to check that if we take one summand of $M\otimes k[y]/(y^m)$ and another of $M\otimes k[y](-j)/(y^n)$, then they are comparable under the partial order $\preceq$ and satisfy the almost centered condition. As in Corollary \ref{tensorisac} with $z=l+y$ and $A=k[z]$ we have:
$$M\otimes k[y](-j_m)/(y^m) \cong \bigoplus_{\alpha \in \Lambda} \bigoplus_{h_{\alpha}=0}^{\min\{d_{\alpha},m\}} A(-(i_{\alpha}+h_{\alpha}+j_m))/(z^{d_{\alpha}+m-2h_{\alpha}-1}),$$ and $$M\otimes k[y](-j_n)/(y^n) \cong \bigoplus_{\beta \in \Lambda} \bigoplus_{h_{\beta}=0}^{\min\{d_{\beta},n\}} A(-(i_{\beta}+h_{\beta}+j_n))/(z^{d_{\beta}+n-2h_{\beta}-1}).$$

First we will see that for all choices of $\alpha$, $h_{\alpha}$, $\beta$ and $h_{\beta}$, $$A(-(i_{\alpha}+h_{\alpha}+j_m))/(z^{d_{\alpha}+m-2h_{\alpha}-1}) \text{~~and~~} A(-(i_{\beta}+h_{\beta}+j_n))/(z^{d_{\beta}+n-2h_{\beta}-1})$$ are comparable under $\preceq$.  Suppose they are not comparable.

If $i_{\alpha}+h_{\alpha}+j_m<i_{\beta}+h_{\beta}+j_n$, then \begin{eqnarray*}
d_{\beta}+n-2h_{\beta}-1+i_{\beta}+h_{\beta}+j_n&>& d_{\alpha}+m-2h_{\alpha}-1+i_{\alpha}+h_{\alpha}+j_m \\
d_{\alpha}+2i_{\alpha}-d_{\beta}-2i_{\beta}&>&m+2jm-n-2j_n+h_{\beta}+i_{\beta}+j_n-h_{\alpha}-i_{\alpha}-j_m\\
d_{\alpha}+2i_{\alpha}-d_{\beta}-2i_{\beta}&>&m+2jm-n-2j_n+1.
\end{eqnarray*} However, $d_{\alpha}+2i_{\alpha}-d_{\beta}-2i_{\beta}=L_{\alpha \beta}-R_{\alpha \beta}=0$ which implies $m+2jm-n-2j_n \leq -2$, contradicting the decomposition of $N$ being almost centered.

If $i_{\beta}+h_{\beta}+j_n<i_{\alpha}+h_{\alpha}+j_m$, then \begin{eqnarray*}
d_{\alpha}+m-2h_{\alpha}-1+i_{\alpha}+h_{\alpha}+j_m&>&d_{\beta}+n-2h_{\beta}-1+i_{\beta}+h_{\beta}+j_n\\
m+2j_m-n-2j_n&>&d_{\beta}+2i_{\beta}-2_{\alpha}-2i_{\alpha}+h_{\alpha}+i_{\alpha}+j_m-h_{\beta}-i_{\beta}-j_n\\
m+2j_m-n-2j_n&>&d_{\beta}+2i_{\beta}-2_{\alpha}-2i_{\alpha}+1.
\end{eqnarray*} However, $d_{\alpha}+2i_{\alpha}-d_{\beta}-2i_{\beta}=L_{\alpha \beta}-R_{\alpha \beta}=0$ which implies $m+2jm-n-2j_n \geq 2$, contradicting the decomposition of $N$ being almost centered.

Hence we have the the decomposition of $M\otimes_k N$ is totally ordered. It remains to show that the decomposition is almost centered, i.e. that $$|d_{\alpha}+m-2h_{\alpha}-1+2(i_{\alpha}+h_{\alpha})-(d_{\beta}+n-2h_{\beta}-1)-2(i_{\beta}+h_{\beta}+j)|\leq 1.$$ However,
\begin{eqnarray*}
&&|d_{\alpha}+m-2h_{\alpha}-1+2(i_{\alpha}+h_{\alpha})-(d_{\beta}+n-2h_{\beta}-1)-2(i_{\beta}+h_{\beta}+j)|\\
&&=|d_{\alpha}+2i_{\alpha}-d_{\beta}-2i_{\beta}+m-n-2j|\\
&&=|m-n-2j|\leq 1.
\end{eqnarray*}
The last equality holds because the Hilbert series of $M$ is symmetric which implies $d_{\alpha}+2i_{\alpha}-d_{\beta}-2i_{\beta}=0$. The last inequality holds because $N$ is almost centered.
\end{proof}

\begin{corollary}\label{symmetrictensor}
Let $l \in R_1$ and $y\in R_1$ be strong Lefschetz elements for the $R$-modules $M$ and $N$ respectively. If the Hilbert functions of $M$ and $N$ are symmetric, then $M\otimes_k N$ has the strong Lefschetz property and a symmetric Hilbert function.
\end{corollary}
\begin{proof}
Theorem \ref{halmostclosedundertensor} tells us that $M\otimes_kN$ has the strong Lefschetz property and is almost centered. We may again assume that $N\cong k[y]/(y^m)\bigoplus k[y](-j)/(y^n)$ with $0\leq j$ and $n+j\leq m$. As in Theorem \ref{halmostclosedundertensor} with $z=l+y$ and $A=k[z]$ we have:
$$M\otimes k[y](-j_m)/(y^m) \cong \bigoplus_{\alpha \in \Lambda} \bigoplus_{h_{\alpha}=0}^{\min\{d_{\alpha},m\}} A(-(i_{\alpha}+h_{\alpha}+j_m))/(z^{d_{\alpha}+m-2h_{\alpha}-1}),$$ and $$M\otimes k[y](-j_n)/(y^n) \cong \bigoplus_{\beta \in \Lambda} \bigoplus_{h_{\beta}=0}^{\min\{d_{\beta},n\}} A(-(i_{\beta}+h_{\beta}+j_n))/(z^{d_{\beta}+n-2h_{\beta}-1}).$$

For all $\alpha$, $\beta$, $0\leq h_{\alpha} \leq m$ and $0 \leq h_{\beta} \leq n$, \begin{eqnarray*}
&&|d_{\alpha}+m-2h_{\alpha}-1+2(i_{\alpha}+h_{\alpha})-(d_{\beta}+n-2h_{\beta}-1)-2(i_{\beta}+h_{\beta}+j)|\\
&&=|d_{\alpha}+2i_{\alpha}-d_{\beta}-2i_{\beta}+m-n-2j|=0,
\end{eqnarray*} where the final equality holds because the Hilbert functions of $M$ and $N$ are symmetric. Hence the Hilbert function of $M\otimes_k N$ is symmetric.
\end{proof}

Before continuing with Theorem \ref{class} mentioned in the introduction we list several lemmas we use repeatedly.

\begin{lemma}\label{pickalpha}
Let $M$ be a finite length graded $R$-module with the strong Lefschetz property and with Hilbert series $H_M(t)=\sum_{i=0}^ph_it^i$. Then 
\begin{itemize}
\item[(i)] $h_i<h_j$ with $i<j$ if and only if there exists $\alpha \in \Lambda$ with $i_{\alpha}>i$ and $d_{\alpha}+i_{\alpha}-1 \geq j$,
\item[(ii)] $h_i<h_j$ with $j<i$ if and only if there exists $\alpha \in \Lambda$ with $i_{\alpha}\leq j$ and $d_{\alpha}+i_{\alpha}-1<i$.
\end{itemize}
\end{lemma}
\begin{proof}
We prove both cases simultaneously. We have $h_i<h_j$ with $i<j$ (respectively $j<i$) if and only if there exists $\alpha \in \Lambda$ that contributes to degree $j$ and not to degree $i$ if and only if $i_{\alpha}>i$ (respectively $i_{\alpha}\leq j$) and $d_{\alpha}+i_{\alpha}-1 \geq j$ (respectively $d_{\alpha}+i_{\alpha}-1<i$).
\end{proof}

\begin{lemma}\label{bounds}
Let $M$ be a finite length graded $R$-module with the strong Lefschetz property. If $M\otimes_k k[y]/(y^m)$ has the strong Lefschetz property for all $m\geq0$, then for all $\alpha \in \Lambda$, $p\leq d_{\alpha}+2i_{\alpha} \leq p+2$.
\end{lemma}
\begin{proof}
This follows immediately from Remark \ref{largestelement} and Theorem \ref{extends} by comparing $S/(l^{p+1})$ and $S(-i_{\alpha})/(l^{d_{\alpha}})$.
\end{proof}

\begin{lemma}\label{Hlemma}
Let $l\in R_1$ be a strong Lefschetz element for $M\cong \bigoplus_{\alpha \in \Lambda}S_{\alpha}$ where $S=k[l]$, such that the decomposition of $M$ as an $S$-module is almost centered.  Then for all $0\leq i \leq \lfloor \frac{p}{2} \rfloor$, either $h_{i-1}\leq h_{p-i}\leq h_i~~\text{or}~~h_{p-i+1}\leq h_i \leq h_{p-i}$.
\end{lemma}
\begin{proof}
Suppose for some $i$ neither sequence of inequalities holds.

If $h_i\leq h_{p-i}$, then $h_i<h_{p-i+1}$ and Lemma \ref{pickalpha} implies there exists $\alpha \in \Lambda$ with $i_{\alpha}>i$ and $d_{\alpha}+i_{\alpha}-1\geq p-i+1$. Putting these together gives $d_{\alpha}+2i_{\alpha}>p+2$, contradicting Lemma \ref{bounds}.

If $h_{p-i}\leq h_i$, then $h_{p-i}<h_{i-1}$ and Lemma \ref{pickalpha} implies there exists $\alpha \in \Lambda$ with $i_{\alpha}\leq i-1$ and $d_{\alpha}+i_{\alpha}-1<p-i$. Putting these together gives $d_{\alpha}+2i_{\alpha}<p$ contradicting Lemma \ref{bounds}.
\end{proof}

\begin{theorem}\label{class}
Let $l\in R_1$ be a strong Lefschetz element for $M$ and $S=k[l]$. Then the following are equivalent:
\begin{itemize}
\item[(i)] $H_M(t)\in \mathcal{H}$, where $\mathcal{H}$ is as in Definition \ref{h},
\item[(ii)] $l+y\in R[y]_1$ is a strong Leschetz element of $M\otimes_k k[y]/(y^m)$ for all $m\geq0$,
\item[(iii)] The decomposition of $M$ as an $S$-module is almost centered.
\end{itemize}
\end{theorem}
\begin{proof} The equivalence of items (ii) and (iii) is established in Theorem \ref{extends}. It suffices to establish the equivalence of items (i) and (iii). 

(iii) implies (i): Suppose the decomposition of $M$ is almost centered. Let $j=\min\{i~|~h_i\neq h_{p-i}\}$. We may assume that $h_j<h_{p-j}$. The minimality of $j$ implies that $h_i=h_{p-i}$ for all $i<j$. Therefore, whenever $i<j$, Lemma \ref{Hlemma} implies $h_{i-1}\leq h_{p-i}=h_i$ or $h_{p-i+1}\leq h_i = h_{p-i}$. However we also have that $h_{i-1}=h_{p-i+1}$, since $i-1<i<j$.  Hence both statements hold and for all $i<j$ and we have $h_{p-i+1}\leq h_i \leq h_{p-i}$ whenever $i\leq j$. So if $H_M(t)\notin {\mathcal {H}}$, then there exists $i>j$ such that $h_{p-i}<h_i$ or $h_{i-1}\leq h_{p-i}=h_i<h_{p-i+1}$.

If $h_{p-i}<h_i$, then Lemma \ref{pickalpha} implies that there exists $\alpha$ such that $i_{\alpha}\leq i$ and $d_{\alpha}+i_{\alpha}-1<p-i$. Combining these yields $d_{\alpha}+2i_{\alpha}<p+1$ and Lemma \ref{bounds} implies $d_{\alpha}+2i_{\alpha}=p+1$. We also have $h_j<h_{p-j}$ which implies that there exists $\beta$ with $i_{\beta}>j$ and $d_{\beta}+i_{\beta}-1\geq p-j$ (Lemma \ref{pickalpha}). Together these give $d_{\beta}+2i_{\beta}>p+1$. Lemma \ref{bounds} implies $d_{\beta}+2i_{\beta}=p+2$. Combining everything gives $|L_{\alpha \beta}-R_{\alpha \beta}|=2$, contradicting the assumption that the decomposition of $M$ is almost centered.

If $h_{i-1}\leq h_{p-i}=h_i<h_{p-i+1}$, then Lemma \ref{pickalpha} gives an $\alpha$ such that $i_{\alpha} >i$ and $d_{\alpha}+i_{\alpha}-1 \geq p-i+1$. Combining these gives $d_{\alpha}+2i_{\alpha}>p+2$, contradicting Lemma \ref{bounds}.

(i) implies (iii): Suppose that the Hilbert series of $M$ is in $\mathcal{H}$ and asuume that the decomposition of $M$ is not almost centered. Choose $\alpha$ and $\beta$ in $\Lambda$ minimally  so that $S_{\alpha}<S_{\beta}$ and $|d_{\beta}+2i_{\beta}-d_{\alpha}-2i_{\alpha}|\geq2$. Here minimally means that for all $\gamma$ and $\delta$ in $\Lambda$ with $S_{\alpha}<S_{\gamma}$ and $S_{\alpha}<S_{\delta}$, $|d_{\gamma}+2i_{\gamma}-d_{\delta}-2i_{\delta}|\leq 1$ and for all $\gamma$ in $\Lambda$ with $S_{\beta}<S_{\gamma}$, $|d_{\gamma}+2i_{\gamma}-d_{\alpha}-2i_{\alpha}|\leq 1$.

We will break this up into the cases where $d_{\beta}+2i_{\beta}-d_{\alpha}-2i_{\alpha}\geq2$ and\\ $d_{\beta}+2i_{\beta}-d_{\alpha}-2i_{\alpha}\leq-2$.

If $d_{\beta}+2i_{\beta}-d_{\alpha}-2i_{\alpha}\leq-2$, let $j=p-d_{\alpha}-i_{\alpha}+1$. Rewriting $j$, we see that 

\begin{eqnarray*}
j&=&(p+1-d_{\beta}-2i_{\beta})+(d_{\beta}+i_{\beta}-d_{\alpha}-i_{\alpha})+i_{\beta} \\
&\leq&1+(d_{\beta}+i_{\beta}-d_{\alpha}-i_{\alpha})+i_{\beta} \\
&\leq&1+i_{\alpha}-i_{\beta}-2+i_{\beta} \\
&=& i_{\alpha}-1.
\end{eqnarray*}
The first inequality follows by comparing $S_{\beta}$ with $S/(l^{p+1})$ and utilizing the minimality of $\alpha$ and $\beta$ (i.e. $|p+1-d_{\beta}-2i_{\beta}|\leq1)$. The second inequality follows because by hypothesis $d_{\beta}+2i_{\beta}-d_{\alpha}-2i_{\alpha}\leq-2$, which after rearranging gives $d_{\beta}+i_{\beta}-d_{\alpha}-i_{\alpha}\leq i_{\alpha}-i_{\beta}-2$. Now, $j\leq i_{\alpha}-1$, and $p-j=d_{\alpha}+i_{\alpha}-1$, so Lemma \ref{pickalpha} gives $h_j<h_{p-j}$.\\

Let $j'=i_{\beta}$. Rewriting $p-j'=p-i_{\beta}$, we see that

\begin{eqnarray*}
p-j'&=& (p+1-d_{\alpha}-2i_{\alpha})+d_{\alpha}+i_{\alpha}-1+(i_{\alpha}-i_{\beta}) \\
&\geq& -1+d_{\alpha}+i_{\alpha}-1+(i_{\alpha}-i_{\beta}) \\
&\geq& d_{\alpha}+i_{\alpha}-2+d_{\beta}+i_{\beta}-d_{\alpha}-i_{\alpha}+2\\
&=&d_{\beta}+i_{\beta}.
\end{eqnarray*}
The first inequality follows by comparing $S_{\alpha}$ with $S/(l^{p+1})$ and utilizing the minimality of $\alpha$ and $\beta$ (i.e. $|p+1-d_{\alpha}-2i_{\alpha}|\leq1$). The second inequality follows because by hypothesis $d_{\beta}+2i_{\beta}-d_{\alpha}-2i_{\alpha}\leq-2$, which after rearranging gives $i_{\beta}-i_{\alpha}\geq d_{\beta}+i_{\beta}-d_{\alpha}-i_{\alpha}+2$. Now $j'=i_{\beta}$ and $p-j'\geq d_{\beta}+i_{\beta}$, therefore $h_{p-j'}<h_{j'}$ by Lemma \ref{pickalpha}.

Finally we need to observe that $1\leq j \leq \lfloor \frac{p}{2} \rfloor$ and $1\leq j' \leq \lfloor \frac{p}{2} \rfloor$. This follows from the observation that $j=i_{\alpha}-1<d_{\alpha}+i_{\alpha}-1=p-j$ and $j'=i_{\beta}<d_{\beta}+i_{\beta}=p-j'$. Thus $H_M(t)\notin {\mathcal{H}}$, a contradiction.

If $d_{\beta}+2i_{\beta}-d_{\alpha}-2i_{\alpha}\geq2$, let $j=i_{\alpha}$. Rewriting $p-j$, we see that

\begin{eqnarray*}
p-j&=&(p+1-d_{\beta}-2i_{\beta})+(i_{\beta}-i_{\alpha})+d_{\beta}+i_{\beta}-1 \\
&\geq& -1+(i_{\beta}-i_{\alpha})+d_{\beta}+i_{\beta}-1 \\
&\geq& -2+d_{\alpha}+i_{\alpha}-d_{\beta}-i_{\beta}+2+d_{\beta}-i_{\beta} \\
&=& d_{\alpha}+i_{\alpha}.
\end{eqnarray*}
The inequalities hold for reasons similar to those discussed in the previous case. We now have that $j=i_{\alpha}$ and $p-j\geq d_{\alpha}+i_{\alpha}$ which implies that $h_{p-j}<h_j$ (Lemma \ref{pickalpha}).\\

Let $j'=p-d_{\beta}-i_{\beta}+1$. Rewriting, we see
\begin{eqnarray*}
j'&=& (p+1-d_{\alpha}-2i_{\alpha})+(d_{\alpha}+i_{\alpha}-d_{\beta}-i_{\beta})+i_{\alpha}\\
&\leq&1+(d_{\alpha}+i_{\alpha}-d_{\beta}-i_{\beta})+i_{\alpha}\\
&\leq& 1+i_{\beta}-i_{\alpha}-2+i_{\alpha}\\
&=&i_{\beta}-1.
\end{eqnarray*}
Hence $j'\leq i_{\beta}-1$ and $h_{p-j'}=d_{\beta}+i_{\beta}-1$, giving that $h_j'<h_{p-j'}$ (Lemma \ref{pickalpha}). As in the last case it is forced that $1\leq j \leq \lfloor \frac{p}{2} \rfloor$ and $1\leq j' \leq \lfloor \frac{p}{2} \rfloor$. Hence  $H_M(t)\notin {\mathcal{H}}$, a contradiction.
\end{proof}

\section{Examples and Applications}

Let $M$ be a graded $R$-module, $l\in R_1$ (not necessarily a strong Lefschetz element for $M$), and set $S=k[l]$. We will use the following diagram to analyze the decomposition of $M$ as a graded $S$-module where $H_M(t)=\sum_{i=o}^ph_it^i$. $$\begin{tabular}{c|c|c|c} $M$& $h_0$ & $\cdots$ & $h_p$ \\\hline  $S(-i_{\alpha})/(l^{d_{\alpha}})$ & $\dim \left[S(-i_{\alpha})/(l^{d_{\alpha}})\right]_0$ & $\cdots$ & $\dim \left[S(-i_{\alpha})/(l^{d_{\alpha}})\right]_p$ \\\hline $\vdots$  & $\vdots$ & $\cdots$ & $\vdots$ \\ \end{tabular}$$ The rows of the diagram are ordered lexicographically. The diagram allows us to tell at a glance whether or not the decomposition is totally ordered or almost centered as in Definitions \ref{order} and \ref{almostcentered} respectively.  If the decomposition is totally ordered than when we look at any two rows, say corresponding to $S_{\alpha}$ and $S_{\beta}$, it will be the case that either $\dim [S_{\alpha}]_i\leq \dim [S_{\beta}]_i$ for all $i$ or $\dim [S_{\beta}]_i\leq \dim [S_{\alpha}]_i$ for all $i$.

\begin{example}
Given $R=k[x_1,x_2]$ and $I=(x_1^2,x_1x_2,x_2^5)$, set $M=R/I$. Let $l=x_2$ and set $S=k[l]$, then $M\cong_{S} S/(l^5)\bigoplus S(-1)/(l)$. The diagram corresponding to the decomposition of $M$ as a graded $S$-module is $$ \begin{tabular}{c|c|c|c|c|c}  & 1 & 2 & 1 & 1 & 1  \\\hline $S/(l^5)$ & 1 & 1 & 1 & 1 & 1  \\\hline $S(-1)/(l)$ & 0 & 1 &  0 & 0 & 0  \\ \end{tabular}.$$Looking at the diagram we can see easily that it is totally ordered, because\\ $\dim [S(-1)/(l)]_i \leq \dim [S/(l^5)]_i$ for all $i$.

If we define $S_{\alpha}:=S/(l^5)$ and $S_{\beta}:=S(-1)/(l)$, then the diagram also makes clear that $|L_{\alpha \beta}-R_{\alpha \beta}|=2$ in this case. Hence the decomposition is not almost centered.  Thus there exists some $m$ such that $M\otimes_k k[y]/(y^m)$ does not have the strong Lefschetz property. For example, if $m=3$, then $M\otimes_kk[y]/(y^3)\cong R[y]/(x_1^2,x_1x_2,x_2^5,y^3)$ does not have the strong Lefschetz property. Let $z=x_2+y$, and $A=k[z]$, then the decomposition of $M\otimes_kk[y]/(y^3)$ as an $A$ module has the following diagram:$$\begin{tabular}{c|c|c|c|c|c|c|c}   & 1 & 3 & 4 & 4 & 3 & 2 & 1  \\\hline  $A/(z^7)$ & 1 & 1 & 1 & 1 & 1 & 1 & 1  \\\hline  $A(-1)/(z^5)$ & 0 & 1 & 1 & 1 & 1 & 1  & 0 \\\hline   $A(-1)/(z^3)$ & 0 & \bf{1} &  1 & 1 & \bf{0} & 0 & 0 \\\hline  $A(-2)/(z^3)$ & 0 & \bf{0} & 1 & 1 &  \bf{1} & 0 & 0 \\ \end{tabular}$$Looking at degrees $1$ and $4$ of the summands $A(-2)/(z^3)$ and $A(-1)/(z^3)$, we see that this decomposition is not totally ordered and hence does not have the strong Lefschetz property. It is an easy exercise to check directly that multiplication by the cube of any linear form fails to be injective or surjective from degree 1 to degree 4 to verify that $A$ fails to have the strong Lefschetz property.  It is worth noting that $A$ does however have the weak Lefschetz property (there exists a linear form $l$ so that multiplication by $l$ is always injective or surjective).
\end{example}

Corollary \ref{tensorisac} and Theorem \ref{halmostclosedundertensor} provide a way of finding more modules with the strong Leschetz property. Corollary \ref{symmetrictensor} shows that $\mathcal{H}$ is closed under tensor when both of the modules have a symmetric Hilbert function.  However, when one module doesn't have a symmetric Hilbert function, the tensor product does not necessarily have the strong Lefschetz property

\begin{example}
Let $R=k[x,y]$, $I=(x^3, x^2y,xy^2,y^4)\subset R$, $J=(x^3,xy,y^4)\subset R$, $M=R/I$ and $N=R/J$. Let $S=k[y]$ and decompose both $M$ and $N$ as graded $S$-modules. The corresponding diagrams show that $y$ is a strong Lefschetz element for both $M$ and $N$.
$$\begin{tabular}{c|c|c|c|c}M & 1 & 2 & 3 & 1 \\\hline $S/(y^4)$ & 1 & 1 & 1 & 1 \\\hline $S(-1)/(y^2)$ & 0 & 1 & 1 & 0 \\\hline $S(-2)/(y)$ & 0 & 0 & 1 & 0\end{tabular}\hspace{2cm} \begin{tabular}{c|c|c|c|c}N & 1 & 2 & 1 & 1 \\\hline $S/(y^4)$ & 1 & 1 & 1 & 1 \\\hline $S(-1)/(y)$ & 0 & 1 & 0 & 0\end{tabular}$$ The decomposition for $M\otimes_kN$ involves the following two rows, where the bolded entries show that $M\otimes_kN$ does not have the strong Lefschetz property: $$\begin{tabular}{c|c|c|c|c|c|c} 0 & \bf{1} & 1 & 1 & 1 & \bf{0} & 0 \\\hline   0 &\bf{0} & 1 & 1 & 1 & \bf{1} & 0 \end{tabular}$$
\end{example}

\begin{remark}\label{monomialcompleteintersection} If $I=(x_1^{a_1},\dots,x_n^{a_n})\subset R=k[x_1,\dots,x_n]$ is a monomial complete intersection, then Reid-Roberts-Roitmann \cite {RRR}, Stanley \cite{S} and Watanabe \cite{W} show that $R/I$ has the strong Lefschetz property. In this case, $R/I\cong \bigotimes_{i=1}^n k[x_i]/(x_i^{a_i})$ and $R/I$ having the strong Lefschetz property also follows from Corollary \ref{symmetrictensor}. This provides as a corollary to our work another purely algebraic proof that monomial complete intersections have the strong Lefschetz property.
\end{remark}

We are also able to use Corollary \ref{symmetrictensor} to say something interesting about the tensor product of cyclic modules when $\dim R =2$.

\begin{corollary} Let $\{I_1,...,I_s\}$ be a set of Artinian ideals in $R=k[x,y]$ such that $R/I_i$ has a symmetric Hilbert function for all $1\leq i\leq s$. Then the tensor product, $\bigotimes_{i=1}^sR/I_i$ has the strong Lefschetz property and a symmetric Hilbert function.
\end{corollary}
\begin{proof}
In \cite[Proposition 4.4]{HMNW} the authors prove that every Artinian ideal in $k[x,y]$ has the strong Lefschetz property. Hence Corollary \ref{symmetrictensor} implies that $\bigotimes_{i=1}^sR/I_i$ has the strong Lefschetz property and a symmetric Hilbert function.
\end{proof}

\section{Positive Characteristic}

Most work on the strong Lefschetz property, as well as on the weak Lefschetz property, has the assumption that the characteristic of the ground field is zero. One reason for such an assumption is that there are simple examples of rings that don't have the strong Lefschetz property when the ground field has positive characteristic.

\begin{example}
Let $R=k[x,y]/(x^p,y^p)$ where $char(k)=p > 0$. If $l$ is any linear form in $R$, then $l^p=0$  Hence $l^p:R_0\to R_p$ is neither injective or surjective showing that $R$ does not have the strong Lefschetz property.
\end{example}

There has been some recent work studying the weak Lefschetz property in the case where the ground field has positive characteristic. In \cite{MMN}, Miglore, Mir\'o-Roig, and Nagel explore the relationship between the weak Lefschetz property and the characteristic of the ground field. In their recent preprint \cite{CN}, Cook and Nagel give several families of examples of rings that have the weak Lefschetz property if the characteristic of the ground field is sufficiently large.

Extensions of some results in this paper to the case where the characteristic of the ground field is positive can be easily seen with the help of the following lemma.

\begin{lemma}\label{charp}
Let $I=(x^a,y^b)\subset R=k[x,y]$ and set $M=R/(x^a,y^b)$ with $\text{char}(k)=p>a+b-2$. Then $M$ has the strong Lefschetz property.
\end{lemma}
\begin{proof}
It suffices to show that $R/\text{Gin}_{\text{revlex}}(I)$ has the strong Lefschetz property. This is because the Hilbert functions of $I+l^a$ and $\text{Gin}_{\text{revlex}}(I)+l^a$ are the same for all $a$ and for $l$ sufficiently general (see proof of \cite[Proposition 4.4]{HMNW}). Notice that $\text{Gin}_{\text{revlex}}(I)$ is Borel fixed. This means that for all monomial generators $m$ of $\text{Gin}_{\text{revlex}}(I)$, if $y^t$ divides $m$ but no higher power of $y$ divides $m$, then $(\frac{x}{y})^sm \in \text{Gin}_{\text{revlex}}(I)$ for all $s<t$ such that ${t \choose s}\equiv 0 \text{ mod } p$ (see \cite{P} or \cite[Proposition 15.23]{E}). Since $p>a+b-2$, this means that $(\frac{x}{y})^sm \in \text{Gin}_{\text{revlex}}(I)$ for all $s<t$. Hence $\text{Gin}_{\text{revlex}}(I)$ is a strongly stable ideal of $k[x,y]$ and therefore a lex ideal. Let $I'\subset k'[x,y]$ be the lex ideal with the same Hilbert function as $I$ and $\text{char}(k')=0$. The Hilbert functions of $I+l^a$, $\text{Gin}_{\text{revlex}}(I)+y^a$ and $I'+y^a$ are all the same. This is because the Hilbert function of monomial ideals is independent of the characteristic of the ground field. Since $\text{char}(k')=0$, we know that $k'[x,y]/I'$ has the strong Lefschetz property \cite[Proposition 4.4]{HMNW} and hence $R/\text{Gin}_{\text{revlex}}(I)$ has the strong Lefschetz property.
\end{proof}

In particular we get a positive characteristic version of Theorem \ref{halmostclosedundertensor}.

\begin{corollary}
Let $M$ and $N$ be Artinian graded $R=k[x_1,\dots,x_n]$ modules with the strong Lefschetz property. Assume that $\text{char}(k)=p>s+t$ where $s$ is the socle degree of $M$ and $t$ is the socle degree of $N$. Let $l \in R_1$ and $y\in R_1$ be strong Lefschetz elements for $M$ and $N$ respectively. If the Hilbert function of $M$ is symmetric and the decomposition of $N$ as a $k[y]$ module is almost centered, then $M\otimes_kN$ has the strong Lefschetz property.
\end{corollary}
\begin{proof}
We know that $M\otimes N \cong \bigoplus k[l,y]/(l^a,y^b)$. Lemma \ref{charp} guarantees that each summand of $M\otimes N$ has the strong Lefschetz property.  It follows from the proof of Theorem \ref{halmostclosedundertensor} that each summand of $M\otimes N$ having the strong Lefschetz property implies that $M\otimes N$ has the strong Lefschetz property.
\end{proof}

\section*{Acknowledgments}
The author would like to thank Giulio Caviglia for introducing her to the strong Lefschetz property and suggesting the problem considered in this paper. The author would also like to thank the referee whose many helpful comments have improved the presentation of the paper.

\bibliographystyle{amsplain}

\end{document}